\documentclass[11pt,a4paper]{article}

\usepackage{latexsym,amsfonts,amsmath,amssymb,mathrsfs,url,color}
\usepackage{graphicx}

\newtheorem{theorem}{Theorem}
\newtheorem{lemma}{Lemma}
\newtheorem{algorithm}{Algorithm}
\newtheorem{remark}{Remark}
\newtheorem{proposition}{Proposition}
\newtheorem{definition}{Definition}
\newtheorem{corollary}{Corollary}

\usepackage[dvipdfm,
linkbordercolor={1 0 0},
citebordercolor={0 1 0},
urlbordercolor={0 1 1}]{hyperref}

\newcommand{\rd}{\,\mathrm{d}}
\newcommand{\rtr}{\,\mathrm{tr}}

\newcommand{\bsx}{\boldsymbol{x}}
\newcommand{\bsy}{\boldsymbol{y}}

\newcommand{\bsl}{\boldsymbol{l}}

\newcommand{\bsk}{\boldsymbol{k}}
\newcommand{\bst}{\boldsymbol{t}}

\newcommand{\bsq}{\boldsymbol{q}}
\newcommand{\bsgamma}{\boldsymbol{\gamma}}

\newcommand{\bszero}{\boldsymbol{0}}
\newcommand{\bsone}{\boldsymbol{1}}

\newcommand{\nat}{\mathbb{N}}
\newcommand{\EE}{\mathbb{E}}

\newcommand{\BB}{\mathcal{B}}

\newcommand{\FF}{\mathbb{F}}

\newcommand{\wal}{\mathrm{wal}}
\newcommand{\cL}{\mathcal{L}}

\newenvironment{proof}{\begin{trivlist}
    \item[\hskip\labelsep{\it Proof.}]}{$\hfill\Box$\end{trivlist}}

\allowdisplaybreaks

\begin{document}

\title{Construction of scrambled polynomial lattice rules over $\mathbb{F}_2$ with small mean square weighted $\mathcal{L}_2$ discrepancy\thanks{The support of Grant-in-Aid for JSPS Fellows No.24-4020 is gratefully acknowledged.}}

\author{Takashi Goda\thanks{Graduate School of Information Science and Technology, The University of Tokyo, 7-3-1 Hongo, Bunkyo-ku, Tokyo 113-8656 ({\tt goda@iba.t.u-tokyo.ac.jp}).}}

\date{\today}

\maketitle

\begin{abstract}
The $\cL_2$ discrepancy is one of several well-known quantitative measures for the equidistribution properties of point sets in the high-dimensional unit cube. The concept of weights was introduced by Sloan and Wo\'{z}niakowski to take into account the relative importance of the discrepancy of lower dimensional projections. As known under the name of quasi-Monte Carlo methods, point sets with small weighted $\cL_2$ discrepancy are useful in numerical integration. This study investigates the component-by-component construction of polynomial lattice rules over the finite field $\FF_2$ whose scrambled point sets have small mean square weighted $\cL_2$ discrepancy. An upper bound on this discrepancy is proved, which converges at almost the best possible rate of $N^{-2+\delta}$ for all $\delta>0$, where $N$ denotes the number of points. Numerical experiments confirm that the performance of our constructed polynomial lattice point sets is comparable or even superior to that of Sobol' sequences.
\end{abstract}
{\em Keywords}:\; Polynomial lattice rules, weighted $\mathcal{L}_2$ discrepancy, numerical integration, randomized quasi-Monte Carlo\\
{\em MSC classifications}:\; 65C05, 65D30, 65D32

%%%%%%%%%%%%%%%%%%%%%%%%%%%%%%%%%%%%%%%%%%%%%%%%%%%%%%%%%%%
%%%%%%%%%%%%%%%%%%%%%%%%%%%%%%%%%%%%%%%%%%%%%%%%%%%%%%%%%%%
%%%%%%%%%%%%%%%%%%%%%%%%%%%%%%%%%%%%%%%%%%%%%%%%%%%%%%%%%%%
\section{Introduction}\label{intro}

In this paper, we study the approximation of an $s$-dimensional integral over the unit cube $[0,1)^s$
  \begin{align*}
    I(f)=\int_{[0,1)^s}f(\bsx)\rd \bsx ,
  \end{align*}
by averaging function evaluations at $N$ points with equal weights
  \begin{align*}
    Q(f)=\frac{1}{N}\sum_{n=0}^{N-1}f(\bsx_n) .
  \end{align*}
Monte Carlo (MC) and quasi-Monte Carlo (QMC) methods choose the point set $P_{N,s}=\{\bsx_0,\ldots, \bsx_{N-1}\}$ randomly and deterministically, respectively. The aim of QMC methods is to distribute the quadrature points as uniformly as possible so as to yield a small integration error. This idea is supported by the general form of various integration error bounds
  \begin{align}\label{eq:error_bound}
    |I(f)-Q(f)|\le V(f)D(P_{N,s}) ,
  \end{align}
where $V(f)$ is the variation of the integrand $f$ in a certain sense, which depends only on $f$, while $D(P_{N,s})$ is the corresponding discrepancy of the point set $P_{N,s}$, which measures the equidistribution properties of $P_{N,s}$ and depends only on $P_{N,s}$. Thus the smaller $D(P_{N,s})$ is, the smaller an integration error we can expect. The most well-known bound of this form is the so-called Koksma-Hlawka inequality in which $V(f)$ is the variation of $f$ in the sense of Hardy and Krause and $D(P_{N,s})$ is the star discrepancy of $P_{N,s}$, see for example \cite{Lem09,Nie92a}.

Randomization of the QMC point set is helpful to obtain statistical information on the integration error and sometimes even enables us to improve the rate of convergence for numerical integration. There have been several methods introduced for randomization \cite{CP76,Hic96,Mat98,Owe95,TF03}. Using the linearity of expectation and (\ref{eq:error_bound}), the mean square integration error is upper-bounded by
  \begin{align*}
    \EE\left[|I(f)-Q(f)|^2\right]\le V^2(f)\EE\left[D^2(\tilde{P}_{N,s})\right] ,
  \end{align*}
where the expectation is taken with respect to all the possible randomized point sets $\tilde{P}_{N,s}$ of $P_{N,s}$. Hence, the mean square discrepancy becomes a meaningful measure of the equidistribution properties of $P_{N,s}$ in this setting.

Among the discrepancy measures, the $\cL_2$ discrepancy is one of the popular measures of the equidistribution properties of point sets. The relationship between the $\cL_2$ discrepancy and numerical integration has been often discussed in the literature, see for example \cite{Hic98,NW10,SW98,Woz91,Zar68}. Sloan and Wo\'{z}niakowski \cite{SW98} introduced the concept of the weighted $\cL_2$ discrepancy to take the relative importance of the discrepancy of lower dimensional projections into account. It provides part of the reason why QMC methods are successful even for very large values of $s$, as often reported in the practical applications to financial problems \cite{CMO97,NT96,PT95}. This phenomenon is hard to explain by the classical integration error bounds. Hence, construction of point sets with small weighted $\cL_2$ discrepancy is of particular interest to practitioners. Especially, in this paper, we focus on constructing randomized QMC point sets with small mean square weighted $\cL_2$ discrepancy.

In order to give the definition of the weighted $\cL_2$ discrepancy, we introduce some notations first. For a point set $P_{N,s}=\{\bsx_0,\ldots, \bsx_{N-1}\}$ in the unit cube $[0,1)^s$, the local discrepancy function is defined as
  \begin{align*}
    \Delta(\bst):=\frac{A_N([\bszero,\bst),P_{N,s})}{N}-t_1\cdots t_s ,
  \end{align*}
where $\bst=(t_1,\ldots, t_s)$ is a vector from $[0,1)^s$, $[\bszero,\bst)$ is the axis-parallel box of the form $[0,t_1)\times \cdots \times [0,t_s)$, and $A_N([\bszero,\bst),P_{N,s})$ denotes the number of indices $n$ with $\bsx_n\in [\bszero,\bst)$. Let $I_s=\{1,\ldots,s\}$ and let $\gamma_u$ be a non-negative real number for $u\subseteq I_s$. We denote by $|u|$ the cardinality of $u$ and by $\bst_u$ a vector from $[0,1)^{|u|}$ containing all the components of $\bst\in [0,1)^s$ whose indices are in $u$. Further, let $\rd \bst_u=\prod_{j\in u}\rd t_j$ and let $(\bst_u,\bsone)$ denote a vector from $[0,1)^s$ with all the components whose indices are not in $u$ replaced by one. Then the weighted $\cL_2$ discrepancy of the point set $P_{N,s}$ is given by
  \begin{align*}
    \cL_{2,N,\bsgamma}(P_{N,s})= \left(\sum_{\emptyset \ne u\subseteq I_s}\gamma_u \int_{[0,1]^{u}}|\Delta(\bst_u,\bsone)|^2 \rd \bst_u\right)^{1/2} .
  \end{align*}
We can recover the classical $\cL_2$ discrepancy by choosing $\gamma_{I_s}=1$ and $\gamma_u=0$ for $u\subset I_s$. The most famous choices of $\gamma_u$ are so-called product weights, that is, $\gamma_u=\prod_{j\in u}\gamma_j$ for all $u\subseteq I_s$. The following proposition generalizes the well-known formula for the classical $\cL_2$ discrepancy introduced by Warnock, see for example \cite{DP10,Mat99}.

\begin{proposition}\label{prop:L2_disc}
For any point set $P_{N,s}=\{\bsx_0,\ldots, \bsx_{N-1}\}$ in $[0,1)^s$ and any sequence $\bsgamma=(\gamma_u)_{u\subseteq I_s}$ of weights, we have
  \begin{align*}
  &  \cL_{2,N,\bsgamma}^2(P_{N,s}) \\
  = & \sum_{\emptyset \ne u\subseteq I_s}\gamma_u\left[ \frac{1}{3^{|u|}}-\frac{2}{N}\sum_{n=0}^{N-1}\prod_{j\in u}\frac{1-x^2_{n,j}}{2}+\frac{1}{N^2}\sum_{n,n'=0}^{N-1}\prod_{j\in u}(1-\max(x_{n,j},x_{n',j}))\right] ,
  \end{align*}
where $x_{n,j}$ is the $j$-th component of the point $\bsx_n$.
\end{proposition}

There are two prominent construction principles of QMC point sets: lattice rules \cite{DKS13,Nie92a,SJ94} and digital $(t,m,s)$-nets \cite{DP10,Nie92a}. In this study, we are concerned with polynomial lattice rules which can be categorized into the latter, while its name comes from the analogy with lattice rules. Since first introduced by Niederreiter \cite{Nie92b}, polynomial lattice rules have been extensively investigated, see for example \cite{DP10,Lec04,Pil12}. In the following, we give the definition of polynomial lattice rules for the case of base 2 because we will only deal with that case.

Let $\FF_2:=\{0,1\}$ be the two element field and denote by $\FF_2((x^{-1}))$ the field of formal Laurent series over $\FF_2$. Every element of $\FF_2((x^{-1}))$ has the form
  \begin{align*}
    L = \sum_{l=w}^{\infty}t_l x^{-l} ,
  \end{align*}
where $w$ is an arbitrary integer and all $t_l\in \FF_2$. Further, we denote by $\FF_2[x]$ the set of all polynomials over $\FF_2$. For a given integer $m$, we define the map $v_m$ from $\FF_2((x^{-1}))$ to the interval $[0,1)$ by
  \begin{align*}
    v_m\left( \sum_{l=w}^{\infty}t_l x^{-l}\right) =\sum_{l=\max(1,w)}^{m}t_l 2^{-l}.
  \end{align*}
We often identify a non-negative integer $k$ whose dyadic expansion is given by $k=\kappa_0+\kappa_1 2+\cdots +\kappa_a 2^a$ with the polynomial $k(x)=\kappa_0+\kappa_1 x+\cdots +\kappa_a x^a \in \FF_2[x]$.  For $\bsk=(k_1,\ldots, k_s)\in (\FF_2[x])^s$ and $\bsq=(q_1,\ldots, q_s)\in (\FF_2[x])^s$, we define the inner product as
  \begin{align*}
     \bsk \cdot \bsq =\sum_{j=1}^{s}k_j q_j \in \FF_2[x] ,
  \end{align*}
and we write $q\equiv 0 \pmod p$ if $p$ divides $q$ in $\FF_2[x]$. Using these notations, the polynomial lattice point set is constructed as follows.

\begin{definition}\label{def:polynomial_lattice}
Let $m, s \in \nat$. Let $p \in \FF_2[x]$ be an irreducible polynomial with $\deg(p)=m$ and let $\bsq=(q_1,\ldots,q_s) \in (\FF_2[x])^s$. The polynomial lattice point set $P_{2^m,s}(\bsq,p)$ is the point set consisting of $2^m$ points given by
  \begin{align*}
    \bsx_n &:= \left( v_m\left( \frac{n(x)q_1(x)}{p(x)} \right) , \ldots , v_m\left( \frac{n(x)q_s(x)}{p(x)} \right) \right) \in [0,1)^s ,
  \end{align*}
for $0\le n<2^m$.
\end{definition}
In the following, the notation $P_{2^m,s}(\bsq,p)$ implicitly means that $\deg(p)=m$ and the number of components for a vector $\bsq$ is $s$.

For randomization of the polynomial lattice point set, we apply Owen's scrambling \cite{Owe95,Owe97a,Owe97b}. It proceeds as follows. For $\bsx=(x_1,\ldots, x_s)\in [0,1)^s$, we denote the dyadic expansion by $x_j=x_{j,1}2^{-1}+x_{j,2}2^{-2}+\cdots$. Let $\bsy=(y_1,\ldots, y_s)\in [0,1)^s$ be the scrambled point of $\bsx$ whose dyadic expansion is represented by $ y_j=y_{j,1}2^{-1}+y_{j,2}2^{-2}+\cdots$. Here we assume that both dyadic expansions of $x_j$ and $y_j$ are unique in the sense that infinitely many digits are different from 1. Each coordinate $y_j$ is obtained by applying permutations to each digit of $x_j$. Here the permutation applied to $x_{j,k}$ depends on $x_{j,l}$ for $1\le l\le k-1$. In particular, $y_{j,1}=\pi_j(x_{j,1}),\, y_{j,2}=\pi_{j,x_{j,1}}(x_{j,2}), y_{j,3}=\pi_{j,x_{j,1},x_{j,2}}(x_{j,3})$, and in general
  \begin{align*}
     y_{j,k}=\pi_{j,x_{j,1},\ldots, x_{j,k-1}}(x_{j,k}) ,
  \end{align*}
where $\pi_{j,x_{j,1},\ldots, x_{j,k-1}}$ is a random permutation of $\{0,1\}$. We choose permutations with different indices mutually independent from each other where each permutation is chosen with the same probability. Then, as shown in \cite[Proposition~2]{Owe95}, the scrambled point $\bsy$ is uniformly distributed in $[0,1)^s$. We refer to \cite{Hic96,Mat98} for simplifications of the above Owen's scrambling algorithm, which can be implemented more easily.

Our aim here is to find a vector $\bsq$ with $p$ fixed, which yields a small mean square weighted $\cL_2$ discrepancy. Restricting each $q_j\in \FF_2[x]$ such that $q_j\ne 0$ and $\deg(q_j)<m$, the number of candidates for $\bsq$ is $(2^{m}-1)^s$, which is quite large. The component-by-component (CBC) construction can significantly reduce the computational burden by searching over all the candidates of $q_{j+1}$ while leaving the existing components ($q_1,\ldots, q_j$) unchanged. The CBC construction was first invented for lattice rules by Korobov \cite{Kor59} and re-discovered more recently by Sloan and Reztsov \cite{SR02}. It also has been applied to polynomial lattice rules. Without requiring exhaustive search, the CBC construction usually finds a good vector $\bsq$ as discussed in many previous studies, see for example \cite{BD11,DKPS05,DLP05,KP07,KP11}. Hence, we employ the CBC construction to find a vector $\bsq$ which gives a small mean square weighted $\cL_2$ discrepancy.

We end this section with a brief outline of this paper. In the next section, we introduce Walsh functions and their useful properties. They play a central role in the analysis of the mean square weighted $\cL_2$ discrepancy. In Section \ref{disc}, we study the mean square weighted $\cL_2$ discrepancy of scrambled polynomial lattice rules. Next, in Section \ref{cbc}, we construct polynomial lattice rules whose scrambled point sets have small mean square weighted $\cL_2$ discrepancy. We consider two cases for weights here: general weights and product weights. Our construction algorithm is extensible in $s$ for product weights, while it is not for general weights. Finding an adequate construction algorithm extensible in $s$ for general weights is open for further research. We prove an upper bound on the root mean square discrepancy which converges at a rate of $N^{-1+\delta}$ for all $\delta>0$, where $N=2^m$ denotes the number of points. As Roth \cite{Rot54} proved that the lower bound on the classical $\cL_2$ discrepancy of $N$ points is given by
  \begin{align}\label{eq:lower_bound}
    \cL_{2,N,\bsgamma}(P_{N,s})\ge c_s\frac{(\log N)^{(s-1)/2}}{N} ,
  \end{align}
where $c_s$ is a constant dependent only on $s$, our upper bound is almost best possible in the sense that a rate of $N^{-1}$ cannot be achieved. We further discuss strong tractability of our construction algorithm. Finally, in Section \ref{numer}, we show the performance of our constructed polynomial lattice point sets and compare with that of the well-known Sobol' sequences.

%%%%%%%%%%%%%%%%%%%%%%%%%%%%%%%%%%%%%%%%%%%%%%%%%%%%%%%%%%%
%%%%%%%%%%%%%%%%%%%%%%%%%%%%%%%%%%%%%%%%%%%%%%%%%%%%%%%%%%%
%%%%%%%%%%%%%%%%%%%%%%%%%%%%%%%%%%%%%%%%%%%%%%%%%%%%%%%%%%%
\section{Walsh functions}\label{walsh}

Walsh functions were first introduced by Walsh \cite{Wal23} and have been extensively studied for example in \cite{Chr55,Fin49}. We refer to \cite[Appendix~A]{DP10} for more information on Walsh functions. In the following, $\nat_0:=\nat \cup \{0\}$ denotes the set of non-negative integers. We first give the definition of dyadic Walsh functions for the one-dimensional case.
\begin{definition}
Let $k\in \nat_0$ with dyadic expansion $k = \kappa_0+\kappa_1 2+\cdots +\kappa_{a}2^{a}$. Then, the $k$-th dyadic Walsh function $\wal_k: [0,1)\to \{-1,1\}$ is defined as
  \begin{align*}
    \wal_k(x) = (-1)^{x_1\kappa_0+\cdots+x_{a+1}\kappa_a} ,
  \end{align*}
for $x\in [0,1)$ with dyadic expansion $x=x_1 2^{-1}+x_2 2^{-2}+\cdots $ (unique in the sense that infinitely many of the $x_i$ are different from 1).
\end{definition}
This definition can be generalized to the higher-dimensional case.

\begin{definition}
For $s\in \nat$, let $\bsx=(x_1,\ldots, x_s)\in [0,1)^s$ and $\bsk=(k_1,\ldots, k_s)\in \nat_0^s$. We define $\wal_{\bsk}: [0,1)^s \to \{-1,1\}$ by
  \begin{align*}
    \wal_{\bsk}(\bsx) = \prod_{j=1}^s \wal_{k_j}(x_j) .
  \end{align*}
\end{definition}

In the following, the operator $\oplus$ denotes the digitwise addition modulo $2$, that is, for $x, y\in [0,1)$ with dyadic representations $x=\sum_{i=1}^{\infty}x_i 2^{-i}$ and $y=\sum_{i=1}^{\infty}y_i 2^{-i}$, $\oplus$ is defined as
  \begin{align*}
    x\oplus y = \sum_{i=1}^{\infty}z_i 2^{-i} ,
  \end{align*}
where $z_i\equiv x_i+y_i \pmod{2}$. Again we assume that the dyadic expansion of $x\oplus y$ is unique in the sense that infinitely many digits are different from 1. We also define a digitwise addition for non-negative integers based on those dyadic representations. In case of vectors in $[0,1)^s$ or $\nat_0^s$, the operator $\oplus$ is carried out componentwise. Further, we call $x\in [0,1)$ a dyadic rational if it can be represented by a finite dyadic expansion. The proposition below summarizes some basic properties of Walsh functions.

\begin{proposition}\label{prop:walsh} We have the following:
\begin{enumerate}
\item For all $k,l\in \nat$ and all $x,y\in [0,1)$ with the restriction that if $x,y$ are not dyadic rationals, then $x\oplus y$ is not allowed to be a dyadic rational, we have
  \begin{align*}
    \wal_k(x)\wal_l(x)=\wal_{k\oplus l}(x) ,\ \wal_k(x)\wal_k(y)=\wal_k(x\oplus y) .
  \end{align*}
\item We have
  \begin{align*}
    \int_{0}^{1}\wal_0(x)\rd x=1 \quad \text{and} \quad \int_{0}^{1}\wal_k(x)\rd x=0 \quad \text{if} \ k\in \nat .
  \end{align*}
\item For all $\bsk, \bsl\in \nat_0^s$, we have
  \begin{align*}
    \int_{[0,1)^s}\wal_{\bsk}(\bsx)\wal_{\bsl}(\bsx)\rd \bsx = \left\{ \begin{array}{ll}
    1 & \text{if} \ \bsk=\bsl ,  \\
    0 & \text{otherwise} .        \\
    \end{array} \right.
  \end{align*}
\item For $s\in \nat$, the system $\{\wal_{\bsk}:\ \bsk=(k_1,\ldots, k_s)\in \nat_0^s\}$ is a complete orthonormal system in $\cL_2([0,1]^s)$.
\end{enumerate}
\end{proposition}

Furthermore, in order to introduce an important relation between Walsh functions and polynomial lattice rules as described below in Lemma \ref{lamma:dual_walsh}, we add one more notation and introduce the concept of the so-called {\it dual polynomial lattice} of a polynomial lattice point set $P_{2^m,s}(\bsq,p)$. For $k\in \nat_0$ with dyadic expansion $k=k_0+k_1 2+\cdots $, $\rtr_m(k)$ gives a polynomial of degree at most $m$ by truncating the associated polynomial $k(x)\in \FF_2[x]$ as
  \begin{align*}
    \rtr_m(k)=k_0+k_1 x+\cdots +k_{m-1}x^{m-1}.
  \end{align*}
For a vector $\bsk=(k_1,\ldots, k_s)\in \nat_0^s$, we define $\rtr_m(\bsk)=(\rtr_m(k_1),\ldots, \rtr_m(k_s))$. With this notation, we introduce the following definition of the dual polynomial lattice $D^*_{\bsq,p}$.
\begin{definition}\label{def:dual_net}
The dual polynomial lattice for a polynomial lattice point set $P_{2^m,s}(\bsq,p)$ is given by
  \begin{align*}
     D^*_{\bsq,p}  = \{ \bsk\in \nat_0^{s}:\ \mathrm{tr}_m(\bsk)\cdot \bsq\equiv 0 \pmod p \} .
  \end{align*}
\end{definition}
Then, the following lemma relates the dual polynomial lattice of a polynomial lattice point set to the numerical integration of Walsh functions. It follows immediately from Definition \ref{def:dual_net}, \cite[Lemma~10.6]{DP10} and \cite[Lemma~4.75]{DP10}.
\begin{lemma}\label{lamma:dual_walsh}
Let $D^*_{\bsq,p}$ be the dual polynomial lattice of a polynomial lattice point set $P_{2^m,s}(\bsq,p)$. Then we have
  \begin{align*}
    \frac{1}{2^m}\sum_{n=0}^{2^m-1}\wal_{\bsk}(\bsx_n)=\left\{ \begin{array}{ll}
     1 & \text{if} \ \bsk\in D^*_{\bsq,p} , \\
     0 & \text{otherwise} . \\
     \end{array} \right.
  \end{align*}
\end{lemma}

%%%%%%%%%%%%%%%%%%%%%%%%%%%%%%%%%%%%%%%%%%%%%%%%%%%%%%%%%%%
%%%%%%%%%%%%%%%%%%%%%%%%%%%%%%%%%%%%%%%%%%%%%%%%%%%%%%%%%%%
%%%%%%%%%%%%%%%%%%%%%%%%%%%%%%%%%%%%%%%%%%%%%%%%%%%%%%%%%%%
\section{Mean square weighted $\cL_2$ discrepancy}\label{disc}

In this section, we study the mean square weighted $\cL_2$ discrepancy of scrambled polynomial lattice rules. In \cite{DPxx}, Dick and Pillichshammer have derived the Walsh series expansion of the classical $\cL_2$ discrepancy. By a slight modification, we can rewrite the expression of the square weighted $\cL_2$ discrepancy given in Proposition \ref{prop:L2_disc} as follows.

\begin{proposition}\label{prop:L2_disc2}
For any point set $P_{N,s}=\{\bsx_0,\ldots, \bsx_{N-1}\}$ in $[0,1)^s$ and any sequence $\bsgamma=(\gamma_u)_{u\subseteq I_s}$ of weights, we have
  \begin{align}
    \cL_{2,N,\bsgamma}^2(P_{N,s}) = \sum_{\emptyset \ne u\subseteq I_s}\gamma_u \sum_{\bsk_u,\bsl_u\in \nat_0^{|u|}\setminus\{\bszero\}}r_u(\bsk_u,\bsl_u)\frac{1}{N^2}\sum_{n,n'=0}^{N-1}\wal_{\bsk_u}(\bsx_{n,u})\wal_{\bsl_u}(\bsx_{n',u}) , \label{eq:L2_disc}
  \end{align}
where $\bsk_u=(k_j)_{j\in u}$, $\bsl_u=(l_j)_{j\in u}$, $r_u(\bsk_u,\bsl_u)=\prod_{j\in u}r(k_j,l_j)$. Further, we have $r(k,l)=r(l,k)$, and for non-negative integers $0\le l\le k$ with dyadic expansions $k=2^{a_1-1}+\cdots+2^{a_v-1}$ with $a_1>\cdots > a_v>0$ and $l=2^{b_1-1}+\cdots+2^{b_w-1}$ with $b_1>\cdots > b_w>0$, we have
  \begin{align*}
    r(k,l) = \left\{ \begin{array}{ll}
    \frac{1}{3}              & \text{if}\ k=l=0 ,                                               \\
    \frac{1}{2^{a_1+2}}      & \text{if}\ v=1\ \text{and}\ l=0 ,                                \\
    -\frac{1}{2^{a_1+a_2+2}} & \text{if}\ v=2\ \text{and}\ l=0 ,                                \\
    -\frac{1}{2^{a_1+a_2+2}} & \text{if}\ v=w+2>2\ \text{and}\ a_3=b_1,\ldots, a_v=b_w ,        \\
    \frac{1}{3\cdot 4^{a_1}} & \text{if}\ k=l>0 ,                                               \\
    \frac{1}{2^{a_1+b_1+2}}  & \text{if}\ v=w, a_1\ne b_1\ \text{and}\ a_2=b_2,\ldots, a_v=b_v ,\\
    0                        & \text{otherwise} .
    \end{array} \right.
  \end{align*}
\end{proposition}

The next corollary provides an expression for the mean square weighted $\cL_2$ discrepancy of scrambled polynomial lattice rules.
\begin{corollary}\label{cor:L2_disc}
For a polynomial lattice point set $P_{2^m,s}(\bsq,p)$, we have
  \begin{align*}
    \EE[\cL_{2,2^m,\bsgamma}^2(\tilde{P}_{2^m,s}(\bsq,p))] =  \sum_{\emptyset \ne v\subseteq I_s}\tilde{\gamma}_v\sum_{\substack{\bsk_v \in \nat^{|v|}\\ (\bsk_v,\bszero)\in D^*_{\bsq,p}}}\psi(\bsk_v,\bszero) ,
  \end{align*}
where we define
  \begin{align*}
    \tilde{\gamma}_v:=\sum_{v \subseteq u\subseteq I_s}\frac{\gamma_u}{3^{|u|}} ,
  \end{align*}
and the expectation is taken with respect to all the possible scrambled point sets $\tilde{P}_{2^m,s}(\bsq,p)$ of $P_{2^m,s}(\bsq,p)$. Further, we denote by $(\bsk_v,\bszero)$ the vector from $\nat_0^s$ with all the components whose indices are not in $v$ replaced by zero, and we have $\psi(k)=1/4^{a_1}$ for $k\in \nat$ with dyadic expansion $k=2^{a_1-1}+\cdots+2^{a_v-1}$ with $a_1>\cdots > a_v>0$, $\psi(0)=1$ and $\psi(\bsk)=\prod_{j=1}^{s}\psi(k_j)$. 
\end{corollary}

\begin{proof}
Let $y,y'\in [0,1)$ be two points obtained by applying Owen's scrambling to the points $x,x'\in [0,1)$. From Owen's lemma \cite[Lemma 13.3]{DP10}, we have
  \begin{align}\label{eq:owen_lemma}
    \EE[\wal_k(y)\wal_l(y')]=0 ,
  \end{align}
whenever $k\ne l$. In the following, we denote by $y_{n,j}$ the point obtained by applying Owen's scrambling to the point $x_{n,j}$. Using (\ref{eq:L2_disc}), (\ref{eq:owen_lemma}), Proposition \ref{prop:walsh} and the linearity of expectation, we have
  \begin{align*}
    & \EE[\cL_{2,2^m,\bsgamma}^2(\tilde{P}_{2^m,s}(\bsq,p))] \\
    = & \sum_{\emptyset \ne u\subseteq I_s}\gamma_u \sum_{\bsk_u,\bsl_u\in \nat_0^{|u|}\setminus\{\bszero\}}r_u(\bsk_u,\bsl_u)\frac{1}{2^{2m}}\sum_{n,n'=0}^{2^m-1}\prod_{j\in u}\EE[\wal_{k_j}(y_{n,j})\wal_{l_j}(y_{n',j})] \\
    = & \sum_{\emptyset \ne u\subseteq I_s}\gamma_u \sum_{\bsk_u\in \nat_0^{|u|}\setminus\{\bszero\}}r_u(\bsk_u,\bsk_u)\frac{1}{2^{2m}}\sum_{n,n'=0}^{2^m-1}\prod_{j\in u}\EE[\wal_{k_j}(y_{n,j}\oplus y_{n',j})] \\
    = & \sum_{\emptyset \ne u\subseteq I_s}\gamma_u \sum_{\emptyset \ne v\subseteq u}\frac{1}{3^{|u\setminus v|}}\sum_{\bsk_v\in \nat^{|v|}}r_v(\bsk_v,\bsk_v)\frac{1}{2^{2m}}\sum_{n,n'=0}^{2^m-1}\prod_{j\in v}\EE[\wal_{k_j}(y_{n,j}\oplus y_{n',j})] .
  \end{align*}

Now we need to introduce the following notations. For $\bsl_v=(l_j)_{j\in v} \in \nat^{|v|}$, we define a set $\BB_{\bsl_v}$ as
  \begin{align*}
     \BB_{\bsl_v}:=\{ (k_j)_{j\in v}\in \nat^{|v|}: 2^{l_j-1} \le k_j< 2^{l_j}\ \text{for}\ j\in v\} .
  \end{align*}
We denote by $\sigma_{\bsl_v}$ the sum of $r_v(\bsk_v,\bsk_v)$ over all $\bsk_v\in \BB_{\bsl_v}$. We have
  \begin{align*}
     \sigma_{\bsl_v} & := \sum_{\bsk_v\in \BB_{\bsl_v}}r_v(\bsk_v,\bsk_v) \; = \; \sum_{\bsk_v\in \BB_{\bsl_v}}\prod_{j\in v}r(k_j,k_j) \\
     & = \prod_{j\in v}\sum_{k_j=2^{l_{j}-1}}^{2^{l_j}-1}r(k_j,k_j) \; = \; \prod_{j\in v}\frac{2^{l_j}-2^{l_{j}-1}}{3\cdot 4^{l_j}} \; = \; \frac{1}{3^{|v|}\cdot 2^{|v|+|\bsl_v|_1}} ,
  \end{align*}
where $|\bsl_v|_1:=\sum_{j\in v}l_j$. Further, we introduce a so-called gain coefficient, which is independent of the choice of $\bsk_v\in \BB_{\bsl_v}$,
  \begin{align*}
    G_{\bsl_v} := \frac{1}{2^{2m}}\sum_{n,n'=0}^{2^m-1}\prod_{j\in v}\EE[\wal_{k_j}(y_{n,j}\oplus y_{n',j})] = 2^{|v|-|\bsl_v|_1}\sum_{\substack{\bsk_v\in \BB_{\bsl_v}\\ (\bsk_v,\bszero)\in D^*_{\bsq,p}}}1 ,
  \end{align*}
where the last equality appeared in the proof of \cite[Corollary~13.7]{DP10}. Using these notations and results, we have
  \begin{align*}
    & \EE[\cL_{2,2^m,\bsgamma}^2(\tilde{P}_{2^m,s}(\bsq,p))] \\
    = & \sum_{\emptyset \ne u\subseteq I_s}\gamma_u \sum_{\emptyset \ne v\subseteq u}\frac{1}{3^{|u\setminus v|}}\sum_{\bsl_v\in \nat^{|v|}}\sum_{\bsk_v\in \BB_{\bsl_v}}r_v(\bsk_v,\bsk_v)\frac{1}{2^{2m}}\sum_{n,n'=0}^{2^m-1}\prod_{j\in v}\EE[\wal_{k_j}(y_{n,j}\oplus y_{n',j})] \\
    = & \sum_{\emptyset \ne u\subseteq I_s}\gamma_u \sum_{\emptyset \ne v\subseteq u}\frac{1}{3^{|u\setminus v|}}\sum_{\bsl_v\in \nat^{|v|}}G_{\bsl_v}\sigma_{\bsl_v} \\
    = & \sum_{\emptyset \ne u\subseteq I_s}\frac{\gamma_u}{3^{|u|}} \sum_{\emptyset \ne v\subseteq u}\sum_{\bsl_v\in \nat^{|v|}}\frac{1}{4^{|\bsl_v|_1}}\sum_{\substack{\bsk_v\in \BB_{\bsl_v}\\ (\bsk_v,\bszero)\in D^*_{\bsq,p}}}1 \\
    = & \sum_{\emptyset \ne u\subseteq I_s}\frac{\gamma_u}{3^{|u|}} \sum_{\emptyset \ne v\subseteq u}\sum_{\substack{\bsk_v \in \nat^{|v|}\\ (\bsk_v,\bszero)\in D^*_{\bsq,p}}}\psi(\bsk_v,\bszero) .
  \end{align*}
The proof is complete by swapping the order of sums.
\end{proof}

We denote the sum in Corollary \ref{cor:L2_disc} by
  \begin{align}\label{eq:L2_disc2}
    B(\bsq,\bsgamma) =  \sum_{\emptyset \ne v\subseteq I_s}\tilde{\gamma}_v \sum_{\substack{\bsk_v \in \nat^{|v|}\\ (\bsk_v,\bszero)\in D^*_{\bsq,p}}}\psi(\bsk_v,\bszero) .
  \end{align}
Using the property of the dual polynomial lattice $D^*_{\bsq,p}$ shown in Lemma \ref{lamma:dual_walsh}, we can derive a more computable form of $B(\bsq,\bsgamma)$. In the following, we write $\log_2$ for the logarithm in base 2 and we set $2^{\lfloor \log_2 0\rfloor}=0$.

\begin{lemma}
Let $B(\bsq,\bsgamma)$ be given by (\ref{eq:L2_disc2}). Then we have
  \begin{align}\label{eq:criterion}
    B(\bsq,\bsgamma) =  \frac{1}{2^m}\sum_{n=0}^{2^m-1}\sum_{\emptyset \ne v\subseteq I_s}\tilde{\gamma}_v\prod_{j\in v}\tilde{\phi}(x_{n,j}) ,
  \end{align}
where for $x\in [0,1)$ we set
  \begin{align*}
    \tilde{\phi}(x) = \frac{1-3\cdot 2^{\lfloor \log_2 x\rfloor}}{2} ,
  \end{align*}
and $\tilde{\gamma}_v$ is defined as in Corollary \ref{cor:L2_disc}. In particular, in case of product weights, we have
  \begin{align}\label{eq:criterion_product}
    B(\bsq,\bsgamma) = -\prod_{j=1}^{s}\left(1+\frac{\gamma_{j}}{3}\right)+\frac{1}{2^m}\sum_{n=0}^{2^m-1}\prod_{j=1}^{s}\left[1+\gamma_{j}\phi (x_{n,j})\right] ,
  \end{align}
where for $x\in [0,1)$ we set
  \begin{align*}
    \phi(x) = \frac{1-2^{\lfloor \log_2 x\rfloor}}{2} .
  \end{align*}
\end{lemma}

\begin{proof}
Applying Lemma \ref{lamma:dual_walsh} to $B(\bsq,\bsgamma)$, we have
  \begin{align*}
    B(\bsq,\bsgamma) & = \sum_{\emptyset \ne v\subseteq I_s}\tilde{\gamma}_v\sum_{\bsk_v \in \nat^{|v|}}\psi(\bsk_v,\bszero)\frac{1}{2^m}\sum_{n=0}^{2^m-1}\wal_{(\bsk_v,\bszero)}(\bsx_n) \\
    & = \frac{1}{2^m}\sum_{n=0}^{2^m-1}\sum_{\emptyset \ne v\subseteq I_s}\tilde{\gamma}_v\prod_{j\in v}\left[ \sum_{k_j=1}^{\infty}\psi(k_j)\wal_{k_j}(x_{n,j})\right] .
  \end{align*}
For the innermost sum, we have by following the similar line as the proof of \cite[Theorem~7.3]{Bal10}
  \begin{align*}
    \sum_{k=1}^{\infty}\psi(k)\wal_{k}(x) = \sum_{l=1}^{\infty}\frac{1}{4^l}\sum_{k=2^{l-1}}^{2^l-1}\wal_{k}(x) = \frac{1-3\cdot 2^{\lfloor \log_2(x)\rfloor}}{2} = \tilde{\phi}(x) .
  \end{align*}
Thus the result for the first part of the lemma follows.

Next in case of $\gamma_v=\prod_{j\in v}\gamma_j$, by letting $\gamma_{\emptyset}=1$, we have
  \begin{align*}
    \tilde{\gamma}_v = \prod_{j\in v}\frac{\gamma_j}{3}\left( \sum_{w\subseteq I_s\setminus v}\prod_{j'\in w}\frac{\gamma_{j'}}{3}\right) = \prod_{j\in v}\frac{\gamma_j}{3}\prod_{j'\in I_s\setminus v}\left(1+\frac{\gamma_{j'}}{3}\right) .
  \end{align*}
Inserting this result into (\ref{eq:criterion}), we have
  \begin{align*}
    B(\bsq,\bsgamma) & = \frac{1}{2^m}\sum_{n=0}^{2^m-1}\sum_{\emptyset \ne v\subseteq I_s}\prod_{j'\in I_s\setminus v}\left(1+\frac{\gamma_{j'}}{3}\right)\prod_{j\in v}\frac{\gamma_j}{3}\tilde{\phi}(x_{n,j}) \\
    & = -\prod_{j=1}^{s}\left(1+\frac{\gamma_{j}}{3}\right)+\frac{1}{2^m}\sum_{n=0}^{2^m-1}\prod_{j=1}^{s}\left[ \left( 1+\frac{\gamma_j}{3}\right)+\frac{\gamma_j}{3}\tilde{\phi}(x_{n,j})\right] \\
    & = -\prod_{j=1}^{s}\left(1+\frac{\gamma_{j}}{3}\right)+\frac{1}{2^m}\sum_{n=0}^{2^m-1}\prod_{j=1}^{s}\left[ 1+\gamma_{j}\phi(x_{n,j})\right].
  \end{align*}
Thus the proof for the second part of the lemma is complete.
\end{proof}

\begin{remark}
Since we have the following recursion in the inner sum of (\ref{eq:criterion})
  \begin{align*}
    \sum_{\emptyset \ne v\subseteq I_r}\tilde{\gamma}_v\prod_{j\in v}\tilde{\phi}(x_{n,j}) = \tilde{\gamma}_{\{r\}}\tilde{\phi}(x_{n,r})+\sum_{\emptyset \ne v\subseteq I_{r-1}}\left( 1+\frac{\tilde{\gamma}_{v\cup \{r\}}}{\tilde{\gamma}_v}\tilde{\phi}(x_{n,r})\right)\tilde{\gamma}_v\prod_{j\in v}\tilde{\phi}(x_{n,j}) ,
  \end{align*}
for $1\le r\le s$, the computational complexity of computing $B(\bsq,\bsgamma)$ with general weights is $O(2^{m+s})$. In case of product weights, on the other hand, the computational complexity of computing $B(\bsq,\bsgamma)$ reduces to $O(s2^m)$.
\end{remark}

%%%%%%%%%%%%%%%%%%%%%%%%%%%%%%%%%%%%%%%%%%%%%%%%%%%%%%%%%%%
%%%%%%%%%%%%%%%%%%%%%%%%%%%%%%%%%%%%%%%%%%%%%%%%%%%%%%%%%%%
%%%%%%%%%%%%%%%%%%%%%%%%%%%%%%%%%%%%%%%%%%%%%%%%%%%%%%%%%%%
\section{Construction of polynomial lattice rules}\label{cbc}

In this section, we first show how to find a vector $\bsq$ by using the CBC construction algorithm for general weights and product weights respectively. We prove that an upper bound on $B(\bsq,\bsgamma)$ satisfied by our algorithm converges at almost the best possible rate of $N^{-2+\delta}$ for all $\delta >0$. Further, in this section we discuss strong tractability of our algorithm.

We restrict each polynomial $q_j$ such that $q_j\ne 0$ and $\deg(q_j)<m$. In the following we denote by $R_m$ the set of all the non-zero polynomials over $\FF_2$ with degree less than $m$, i.e.,
  \begin{align*}
    R_m = \{q\in \FF_2[x]: \deg(q)<m \ \text{and}\ q\ne 0\} .
  \end{align*}
It is clear that $|R_m|=2^m-1$. We write $\bsq_\tau=(q_1,\ldots, q_\tau)$ for $1\le \tau\le s$.

%%%%%%%%%%%%%%%%%%%%%%%%%%%%%%%%%%%%%%%%%%%%%%%%%%%%%%%%%%%
\subsection{General weights}
The CBC construction for general weights proceeds as follows.

\begin{algorithm}\label{cbc_algor}(CBC construction for general weights)
For $m,s\in \nat$ and any sequence of weights $\bsgamma=(\gamma_u)_{u\subseteq I_s}$, we proceed as follows.
	\begin{enumerate}
		\item Choose an irreducible polynomial $p\in \FF_2[x]$ with $\deg(p)=m$.
		\item Set $q_1^*=1$.
		\item For $\tau=2,\ldots, s$, find $q^*_{\tau}$ by minimizing $B((\bsq^{*}_{\tau-1},q_\tau),\bsgamma)$ as a function of $q_\tau\in R_{m}$
		where 
  \begin{align}\label{eq:cbc_criterion_general}
    B((\bsq^{*}_{\tau-1},q_\tau),\bsgamma) = & \sum_{\emptyset \ne v\subseteq I_\tau}\tilde{\gamma}_v \sum_{\substack{\bsk_v \in \nat^{|v|}\\ (\bsk_v,\bszero)\in D^*_{(\bsq^{*}_{\tau-1},q_\tau),p}}}\psi(\bsk_v,\bszero) \\
    = & \frac{1}{2^m}\sum_{n=0}^{2^m-1}\sum_{\emptyset \ne v\subseteq I_\tau}\tilde{\gamma}_v\prod_{j\in v}\tilde{\phi}(x_{n,j}) \nonumber ,
  \end{align}
        in which $\tilde{\gamma}_v$ is defined as in Corollary \ref{cor:L2_disc}.
	\end{enumerate}
\end{algorithm}

\begin{remark}
Since computing $\tilde{\gamma}_v$ in Step 3. of Algorithm \ref{cbc_algor} requires $\gamma_v$ such that $v\nsubseteq I_\tau$, Algorithm \ref{cbc_algor} is not extensible in $s$. A similar situation occurs for lattice rules as has been discussed in \cite[Chapter~5.4]{DKS13}. From the last line in the proof of Corollary \ref{cor:L2_disc}, it is possible to replace $B((\bsq^{*}_{\tau-1},q_\tau),\bsgamma)$ by
\begin{align*}
    B((\bsq^{*}_{\tau-1},q_\tau),\bsgamma) & =  \sum_{\emptyset \ne u\subseteq I_\tau}\frac{\gamma_u}{3^{|u|}} \sum_{\emptyset \ne v\subseteq u}\sum_{\substack{\bsk_v \in \nat^{|v|}\\ (\bsk_v,\bszero)\in D^*_{(\bsq^{*}_{\tau-1},q_\tau),p}}}\psi(\bsk_v,\bszero) .
  \end{align*}
Then we can make the construction algorithm extensible in $s$. Because of the technical difficulty in treating two outermost sums in the right-hand side, however, it is hard to prove that polynomial lattice rules constructed using this replaced criterion achieve almost the best possible rate of convergence.
\end{remark}

The next theorem provides an upper bound on $B(\bsq_\tau,\bsgamma)$ for the polynomials $\bsq^{*}_\tau$ for $1\le \tau\le s$ constructed according to Algorithm \ref{cbc_algor}. It converges at almost the best possible rate of $N^{-2+\delta}$ for all $\delta >0$. In the proof of the theorem, we use the following inequality, which states that for any sequence $(a_i)_{i\in \nat}$ of non-negative real numbers we have
  \begin{align}\label{eq:jensen}
    \left( \sum a_i\right)^{\lambda} \le \sum a_i^{\lambda} ,
  \end{align}
for any $0<\lambda \le 1$.

\begin{theorem}\label{theorem1}
Let $p\in \FF_2[x]$ be an irreducible polynomial with $\deg(p)=m$. Suppose that $\bsq^{*}_s\in R_m^s$ is constructed according to Algorithm \ref{cbc_algor}. Then for all $\tau=1,\ldots, s$ we have
  \begin{align}\label{eq:theorem1}
    B(\bsq^{*}_\tau,\bsgamma) \le \frac{1}{(2^m-1)^{1/\lambda}}\left[ \sum_{\emptyset \ne v\subseteq I_\tau}\tilde{\gamma}_v^{\lambda} \frac{1}{(2^{2\lambda}-2)^{|v|}} \right]^{1/\lambda} ,
  \end{align}
for $1/2<\lambda \le 1$, where $\tilde{\gamma}_v$ is defined as in Corollary \ref{cor:L2_disc}.
\end{theorem}

\begin{proof}
We prove the theorem by induction on $\tau$. For $\tau=1$, we have
  \begin{align*}
     B(q^{*}_1,\bsgamma) & = \tilde{\gamma}_{\{1\}}\sum_{\substack{k=1\\ 2^m\mid k}}^{\infty}\psi(k) \; = \; \tilde{\gamma}_{\{1\}}\sum_{a=1}^{\infty}\sum_{\substack{k=2^{a-1}\\ 2^m\mid k}}^{2^a-1}\psi(k) \\
     & = \tilde{\gamma}_{\{1\}}\sum_{a=m+1}^{\infty}2^{a-m-1}\cdot 2^{-2a} \; = \; \tilde{\gamma}_{\{1\}}\frac{1}{2^{2m+1}} \; \le \; \tilde{\gamma}_{\{1\}}\left[\frac{1}{(2^{m}-1)(2^{2\lambda}-2)}\right]^{1/\lambda} ,
  \end{align*}
for $1/2<\lambda \le 1$. Hence the result holds true for $\tau=1$.

Next, assume that the statement of the theorem is true for some $\tau\ge 1$. Then it is enough to show that the statement is also true for the ($\tau+1$)-th component. In the following, we classify each subset $u$ according to whether $u$ includes the component $\{\tau+1\}$ or not. Then we have
  \begin{align*}
    & B((\bsq^{*}_\tau,q_{\tau+1}),\bsgamma) \\
    = & \sum_{\emptyset \ne v\subseteq I_{\tau+1}}\tilde{\gamma}_v \sum_{\substack{\bsk_v \in \nat^{|v|}\\ (\bsk_v,\bszero)\in D^*_{(\bsq^{*}_\tau,q_{\tau+1}),p}}}\psi(\bsk_v,\bszero) \\
    = & \sum_{\emptyset \ne v\subseteq I_{\tau}}\tilde{\gamma}_v \sum_{\substack{\bsk_v \in \nat^{|v|}\\ (\bsk_v,\bszero)\in D^*_{\bsq^{*}_\tau,p}}}\psi(\bsk_v,\bszero) + \sum_{v\subseteq I_{\tau}}\tilde{\gamma}_{v\cup \{\tau+1\}} \sum_{\substack{(\bsk_v,k_{\tau+1}) \in \nat^{|v|+1}\\ (\bsk_v,k_{\tau+1},\bszero)\in D^*_{(\bsq^{*}_\tau,q_{\tau+1}),p}}}\psi(\bsk_v,k_{\tau+1},\bszero) \\
    = & B(\bsq^{*}_\tau,\bsgamma)+\theta(q_{\tau+1}) ,
  \end{align*}
where we have defined
  \begin{align*}
    \theta(q_{\tau+1}) := \sum_{v\subseteq I_{\tau}}\tilde{\gamma}_{v\cup \{\tau+1\}} \sum_{\substack{(\bsk_v,k_{\tau+1}) \in \nat^{|v|+1}\\ (\bsk_v,k_{\tau+1},\bszero)\in D^*_{(\bsq^{*}_\tau,q_{\tau+1}),p}}}\psi(\bsk_v,k_{\tau+1},\bszero) .
  \end{align*}

In order to obtain an upper bound on $\theta(q_{\tau+1}^{*})$, we employ the averaging argument. Since we choose $q_{\tau+1}^*$ which minimizes $\theta(q_{\tau+1})$ in Algorithm \ref{cbc_algor}, $\theta^\lambda(q_{\tau+1}^{*})$ has to be less than or equal to the average of $\theta^\lambda(q_{\tau+1})$ over $q_{\tau+1}\in R_{m}$ for any $1/2< \lambda \le 1$. We obtain
  \begin{align*}
    \theta^{\lambda}(q^{*}_{\tau+1}) & \le \frac{1}{2^m-1}\sum_{q_{\tau+1}\in R_{m}}\theta^{\lambda}(q_{\tau+1}) \\
    & \le \frac{1}{2^m-1}\sum_{q_{\tau+1}\in R_{m}}\sum_{v\subseteq I_{\tau}}\tilde{\gamma}^\lambda_{v\cup \{\tau+1\}} \sum_{\substack{(\bsk_v,k_{\tau+1}) \in \nat^{|v|+1}\\ (\bsk_v,k_{\tau+1},\bszero)\in D^*_{(\bsq^{*}_\tau,q_{\tau+1}),p}}}\psi^\lambda(\bsk_v,k_{\tau+1},\bszero) \\
    & = \sum_{v\subseteq I_\tau}\frac{\tilde{\gamma}^\lambda_{v\cup \{\tau+1\}}}{2^m-1}\sum_{q_{\tau+1}\in R_{m}} \sum_{\substack{(\bsk_v,k_{\tau+1}) \in \nat^{|v|+1}\\ (\bsk_v,k_{\tau+1},\bszero)\in D^*_{(\bsq^{*}_\tau,q_{\tau+1}),p}}}\psi^\lambda(\bsk_v,k_{\tau+1},\bszero) ,
  \end{align*}
where we have used (\ref{eq:jensen}) in the second inequality. For a fixed $v\subseteq I_\tau$, we consider the condition $(\bsk_v,k_{\tau+1},\bszero)\in D^*_{(\bsq^{*}_\tau,q_{\tau+1}),p}$, that is,
  \begin{align*}
    \rtr_m(\bsk_v)\cdot \bsq_v + \rtr_m(k_{\tau+1})\cdot q_{\tau+1} \equiv 0 \pmod p .
  \end{align*}
If $k_{\tau+1}$ is a multiple of $2^m$, we always have $\rtr_m(k_{\tau+1})=0$ and the above equation becomes independent of $q_{\tau+1}$. Otherwise if $k_{\tau+1}$ is not a multiple of $2^m$, we have $\rtr_m(k_{\tau+1})\ne 0$ and the term $\rtr_m(k_{\tau+1})\cdot q_{\tau+1}$ cannot be a multiple of $p$. Thus we have
  \begin{align*}
    & \frac{1}{2^m-1}\sum_{q_{\tau+1}\in R_{m}} \sum_{\substack{(\bsk_v,k_{\tau+1}) \in \nat^{|v|+1}\\ (\bsk_v,k_{\tau+1},\bszero)\in D^*_{(\bsq^{*}_\tau,q_{\tau+1}),p}}}\psi^\lambda(\bsk_v,k_{\tau+1},\bszero) \\
    = & \sum_{\substack{k_{\tau+1}=1\\ 2^m\mid k_{\tau+1}}}^{\infty}\psi^{\lambda}(k_{\tau+1})\sum_{\substack{\bsk_v\in \nat^{|v|}\\ \rtr_m(\bsk_v)\cdot \bsq_v\equiv 0 \pmod p}}\psi^{\lambda}(\bsk_v) \\
    & + \frac{1}{2^m-1}\sum_{\substack{k_{\tau+1}=1\\ 2^m\nmid k_{\tau+1}}}^{\infty}\psi^{\lambda}(k_{\tau+1})\sum_{\substack{\bsk_v\in \nat^{|v|}\\ \rtr_m(\bsk_v)\cdot \bsq_v\not\equiv 0 \pmod p}}\psi^{\lambda}(\bsk_v) \\
    \le & \frac{1}{2^{2\lambda m}}\sum_{k_{\tau+1}=1}^{\infty}\psi^{\lambda}(k_{\tau+1})\sum_{\substack{\bsk_v\in \nat^{|v|}\\ \rtr_m(\bsk_v)\cdot \bsq_v\equiv 0 \pmod p}}\psi^{\lambda}(\bsk_v) \\
    & + \frac{1}{2^m-1}\sum_{k_{\tau+1}=1}^{\infty}\psi^{\lambda}(k_{\tau+1})\sum_{\substack{\bsk_v\in \nat^{|v|}\\ \rtr_m(\bsk_v)\cdot \bsq_v\not\equiv 0 \pmod p}}\psi^{\lambda}(\bsk_v) \\
    \le & \frac{1}{2^m-1}\sum_{k_{\tau+1}=1}^{\infty}\psi^{\lambda}(k_{\tau+1})\sum_{\bsk_v\in \nat^{|v|}}\psi^{\lambda}(\bsk_v) \\
    = & \frac{1}{2^m-1}\left[ \sum_{k=1}^{\infty}\psi^{\lambda}(k)\right]^{|v|+1}\; = \; \frac{1}{(2^m-1)(2^{2\lambda}-2)^{|v|+1}} .
  \end{align*}
Using this result, we obtain an upper bound on $\theta(q_{\tau+1}^{*})$ as
  \begin{align*}
    \theta^{\lambda}(q^{*}_{\tau+1}) & \le \frac{1}{2^m-1}\sum_{v\subseteq I_\tau}\tilde{\gamma}^\lambda_{v\cup \{\tau+1\}}\frac{1}{(2^{2\lambda}-2)^{|v|+1}} .
  \end{align*}
Finally by applying (\ref{eq:jensen}) we have 
  \begin{align*}
    & B^{\lambda}((\bsq^{*}_{\tau},q_{\tau+1}^{*}),\bsgamma) \\
    = & \left( B(\bsq^{*}_{\tau},\bsgamma)+\theta(q_{\tau+1}^{*}) \right)^{\lambda} \\
    \le & B^{\lambda}(\bsq^{*}_{\tau},\bsgamma)+\theta^{\lambda}(q_{\tau+1}^{*}) \\
    \le & \frac{1}{2^m-1}\sum_{\emptyset \ne v\subseteq I_\tau}\tilde{\gamma}_v^{\lambda} \frac{1}{(2^{2\lambda}-2)^{|v|}} +\frac{1}{2^m-1}\sum_{v\subseteq I_\tau}\tilde{\gamma}^\lambda_{v\cup \{\tau+1\}}\frac{1}{(2^{2\lambda}-2)^{|v|+1}}  \\
    = & \frac{1}{2^m-1}\sum_{\emptyset \ne v\subseteq I_{\tau+1}}\tilde{\gamma}_v^{\lambda} \frac{1}{(2^{2\lambda}-2)^{|v|}} ,
  \end{align*}
for $1/2< \lambda \le 1$, from which (\ref{eq:theorem1}) holds true for the $(\tau+1)$-th component. Hence the result follows.
\end{proof}

\begin{remark}\label{remark:cbc_one_dimension}
For $\tau=1$, we have as in the proof of Theorem \ref{theorem1}
  \begin{align*}
     B(q^{*}_1,\bsgamma) = \tilde{\gamma}_{\{1\}}\frac{1}{2^{2m+1}}.
  \end{align*}
Since the lower bound on the $\cL_2$ discrepancy is given as in (\ref{eq:lower_bound}), this achieves the best possible rate of convergence. As a one-dimensional polynomial lattice point set consists of the equidistributed points $x_n=n/b^m$, $n=0,\ldots, 2^m-1$, other QMC point sets such as Sobol' and Niederreiter sequences constructed over $\FF_2$ also give the same result.
\end{remark}

\begin{remark}\label{remark:cbc_bound}
For $\tau=s$, we further have
  \begin{align*}
     B(q^{*}_s,\bsgamma) & \le \frac{1}{(2^m-1)^{1/\lambda}}\left[ \sum_{\emptyset \ne v\subseteq I_s}\tilde{\gamma}_v^{\lambda} \frac{1}{(2^{2\lambda}-2)^{|v|}} \right]^{1/\lambda} \\
     & \le \frac{1}{(2^m-1)^{1/\lambda}}\left[ \sum_{\emptyset \ne v\subseteq I_s}\left(\sum_{v\subseteq u\subseteq I_s}\left( \frac{\gamma_u}{3^{|u|}}\right)^{\lambda} \right) \frac{1}{(2^{2\lambda}-2)^{|v|}} \right]^{1/\lambda} \\
     & = \frac{1}{(2^m-1)^{1/\lambda}}\left[ \sum_{\emptyset \ne v\subseteq I_s}\left( \frac{\gamma_v}{3^{|v|}}\right)^{\lambda} \left( -1+\left(\frac{2^{2\lambda}-1}{2^{2\lambda}-2}\right)^{|v|}\right) \right]^{1/\lambda} ,
  \end{align*}
where we have used (\ref{eq:jensen}) in the second inequality and swapped the order of sums in the last equality. Thus, we have obtained an upper bound on $B(q^{*}_s,\bsgamma)$ using not $\tilde{\gamma}_v$ but $\gamma_v$ for $v\subseteq I_s$.
\end{remark}

In the following we discuss strong tractability of our construction algorithm. Let us consider the inverse of the mean square weighted $\cL_2$ discrepancy which is defined as follows
  \begin{align*}
    N(s,\epsilon)=\min\{N\in \nat: \EE[\cL_{2,N,\bsgamma}^2(\tilde{P}_{N,s})]\le \epsilon \EE[\cL_{2,0,\bsgamma}^2(\tilde{P}_{0,s})]\} .
  \end{align*}
We say that the mean square weighted $\cL_2$ discrepancy is strongly tractable if there exist non-negative constants $C$ and $\beta$ such that
  \begin{align*}
    N(s,\epsilon)\le C\epsilon^{-\beta} ,
  \end{align*}
where $C$ depends neither on $\epsilon$ or $s$ and we call $\beta$ the exponent of tractability.

In the next corollary, we write $\gamma_{s,u}$ instead of $\gamma_u$ to emphasize the dependence on $s$ of our construction algorithm. We denote by $\bsgamma$ a sequence of weights $(\gamma_{s,u})_{u\subseteq I_s}$ for $s\in \nat$.

\begin{corollary}\label{cor:tractability}
Assume that the weights $\bsgamma$ satisfy the condition
  \begin{align*}
    B_{\bsgamma,\lambda} := \sup_{s\in \nat}\frac{\left[\sum_{\emptyset \ne u\subseteq I_s}\left( \frac{\gamma_{s,u}}{3^{|u|}}\right)^{\lambda}\left(-1+\left( \frac{2^{2\lambda}-1}{2^{2\lambda}-2}\right)^{|u|} \right)\right]^{1/\lambda}}{\sum_{\emptyset \ne u\subseteq I_s}\frac{\gamma_{s,u}}{3^{|u|}}} < \infty ,
  \end{align*}
for some $\lambda$ such that $1/2<\lambda \le 1$. Then the mean square weighted $\cL_2$ discrepancy is strongly tractable with the exponent of tractability at most $\lambda$.
\end{corollary}

\begin{proof}
For the empty point set $P_{0,s}$, we have
  \begin{align*}
    \EE[\cL_{2,0,\bsgamma}^2(\tilde{P}_{0,s})] & = \sum_{\emptyset \ne u\subseteq I_s}\gamma_{s,u} \prod_{j\in u}\int_{0}^{1}t_j^2 \rd t_j \\
    & = \sum_{\emptyset \ne u\subseteq I_s}\frac{\gamma_{s,u}}{3^{|u|}} .
  \end{align*}
For a polynomial lattice point set $P_{2^m,s}$ constructed by Algorithm \ref{cbc_algor}, we have from Remark \ref{remark:cbc_bound}
  \begin{align*}
    \EE[\cL_{2,2^m,\bsgamma}^2(P_{2^m,s})] & \le \frac{1}{(2^m-1)^{1/\lambda}}\left[ \sum_{\emptyset \ne u\subseteq I_s}\left( \frac{\gamma_{s,u}}{3^{|u|}}\right)^{\lambda}\left(-1+\left( \frac{2^{2\lambda}-1}{2^{2\lambda}-2}\right)^{|u|} \right)\right]^{1/\lambda} \\
    & \le \frac{1}{(2^m-1)^{1/\lambda}}B_{\bsgamma,\lambda}\sum_{\emptyset \ne u\subseteq I_s}\frac{\gamma_{s,u}}{3^{|u|}}  \\
    & = \frac{1}{(2^m-1)^{1/\lambda}}B_{\bsgamma,\lambda}\EE[\cL_{2,0,\bsgamma}^2(P_{0,s})] .
  \end{align*}
The last term is smaller than or equal to $\epsilon \EE[\cL_{2,0,\bsgamma}^2(P_{0,s})]$ if $N=2^m\ge 1+B^{\lambda}_{\bsgamma,\lambda}\epsilon^{-\lambda}$. Thus the result follows.
\end{proof}

%%%%%%%%%%%%%%%%%%%%%%%%%%%%%%%%%%%%%%%%%%%%%%%%%%%%%%%%%%%
\subsection{Product weights}
In case of product weights, we have for (\ref{eq:cbc_criterion_general})
  \begin{align*}
    & B((\bsq^{*}_{\tau-1},q_\tau),\bsgamma) \\
    = & \sum_{\emptyset \ne v\subseteq I_\tau}\prod_{j\in v}\frac{\gamma_j}{3}\prod_{j'\in I_s\setminus v}\left(1+\frac{\gamma_{j'}}{3}\right) \sum_{\substack{\bsk_v \in \nat^{|v|}\\ (\bsk_v,\bszero)\in D^*_{(\bsq^{*}_{\tau-1},q_\tau),p}}}\psi(\bsk_v,\bszero) \\
    = & \prod_{j''=\tau+1}^{s}\left(1+\frac{\gamma_{j''}}{3}\right)\sum_{\emptyset \ne v\subseteq I_\tau}\prod_{j\in v}\frac{\gamma_j}{3}\prod_{j'\in I_\tau\setminus v}\left(1+\frac{\gamma_{j'}}{3}\right) \sum_{\substack{\bsk_v \in \nat^{|v|}\\ (\bsk_v,\bszero)\in D^*_{(\bsq^{*}_{\tau-1},q_\tau),p}}}\psi(\bsk_v,\bszero) \\
    = & \prod_{j''=\tau+1}^{s}\left(1+\frac{\gamma_{j''}}{3}\right)\left[-\prod_{j=1}^{\tau}\left(1+\frac{\gamma_{j}}{3}\right)+\frac{1}{2^m}\sum_{n=0}^{2^m-1}\prod_{j=1}^{\tau}\left[1+\gamma_{j}\phi (x_{n,j})\right]\right] .
  \end{align*}
Omitting the term $\prod_{j''=\tau+1}^{s}\left(1+\frac{\gamma_{j''}}{3}\right)$ from the criterion, computing $B((\bsq^{*}_{\tau-1},q_\tau),\bsgamma)$ does not require $\gamma_{\tau+1},\ldots,\gamma_s$. Therefore we can make the construction algorithm extensible in $s$, while it is still possible to prove that the constructed polynomial lattice rules achieve almost the best possible rate of convergence as shown below in Theorem \ref{theorem2}. Thus the CBC construction for product weights proceeds as follows.

\begin{algorithm}\label{cbc_algor_product}(CBC construction for product weights)
For $m,s\in \nat$ and product weights $\gamma_u=\prod_{j\in u}\gamma_j$ with any non-negative numbers $\gamma_1,\ldots,\gamma_s$, we proceed as follows.
	\begin{enumerate}
		\item Choose an irreducible polynomial $p\in \FF_2[x]$ with $\deg(p)=m$.
		\item Set $q_1^*=1$.
		\item For $\tau=2,\ldots, s$, find $q^*_{\tau}$ by minimizing $B((\bsq^{*}_{\tau-1},q_\tau),\bsgamma)$ as a function of $q_\tau\in R_{m}$
		where 
  \begin{align*}
    B((\bsq^{*}_{\tau-1},q_\tau),\bsgamma) = & \sum_{\emptyset \ne v\subseteq I_\tau}\tilde{\gamma}_{\tau,v} \sum_{\substack{\bsk_v \in \nat^{|v|}\\ (\bsk_v,\bszero)\in D^*_{(\bsq^{*}_{\tau-1},q_\tau),p}}}\psi(\bsk_v,\bszero) \\
    = & -\prod_{j=1}^{\tau}\left(1+\frac{\gamma_{j}}{3}\right)+\frac{1}{2^m}\sum_{n=0}^{2^m-1}\prod_{j=1}^{\tau}\left[1+\gamma_{j}\phi (x_{n,j})\right] ,
  \end{align*}
        in which we define
  \begin{align*}
    \tilde{\gamma}_{\tau,v}:=\prod_{j\in v}\frac{\gamma_j}{3}\prod_{j'\in I_\tau\setminus v}\left(1+\frac{\gamma_{j'}}{3}\right).
  \end{align*}
	\end{enumerate}
\end{algorithm}

\begin{theorem}\label{theorem2}
Let $p\in \FF_2[x]$ be irreducible polynomial with $\deg(p)=m$. Suppose that $\bsq^{*}_s\in R_m^s$ is constructed according to Algorithm \ref{cbc_algor_product}. Then for any $\tau=1,\ldots, s$ we have
  \begin{align*}
    B(\bsq^{*}_\tau,\bsgamma) \le \frac{1}{(2^m-1)^{1/\lambda}}\left[ \sum_{\emptyset \ne v\subseteq I_\tau}\tilde{\gamma}_{\tau,v}^{\lambda} \frac{1}{(2^{2\lambda}-2)^{|v|}} \right]^{1/\lambda} ,
  \end{align*}
for $1/2<\lambda \le 1$, where $\tilde{\gamma}_{\tau,v}$ is defined as in Algorithm \ref{cbc_algor_product}.
\end{theorem}

\begin{proof}
We prove the theorem by induction on $\tau$ in a quite similar way as the proof of Theorem \ref{theorem1}. For $\tau=1$, we have the result by replacing $\tilde{\gamma}_{\{1\}}$ with $\gamma_{\{1\}}/3$. For $\tau\ge 1$, assume that the statement of the theorem is true. Then we have 
  \begin{align*}
    & B((\bsq^{*}_\tau,q_{\tau+1}),\bsgamma) \\
    = & \sum_{\emptyset \ne v\subseteq I_{\tau+1}}\tilde{\gamma}_{\tau+1,v} \sum_{\substack{\bsk_v \in \nat^{|v|}\\ (\bsk_v,\bszero)\in D^*_{(\bsq^{*}_{\tau},q_{\tau+1}),p}}}\psi(\bsk_v,\bszero) \\
    = & \left(1+\frac{\gamma_{\tau+1}}{3}\right)\sum_{\emptyset \ne v\subseteq I_\tau}\tilde{\gamma}_{\tau,v} \sum_{\substack{\bsk_v \in \nat^{|v|}\\ (\bsk_v,\bszero)\in D^*_{\bsq^{*}_\tau,p}}}\psi(\bsk_v,\bszero) \\
    & + \sum_{v\subseteq I_\tau}\tilde{\gamma}_{\tau+1,v\cup \{\tau+1\}} \sum_{\substack{(\bsk_v,k_{\tau+1}) \in \nat^{|v|+1}\\ (\bsk_v,k_{\tau+1},\bszero)\in D^*_{(\bsq^{*}_{\tau},q_{\tau+1}),p}}}\psi(\bsk_v,k_{\tau+1},\bszero) \\
    = & \left(1+\frac{\gamma_{\tau+1}}{3}\right)B(\bsq^{*}_\tau,\bsgamma)+\theta(q_{\tau+1}) ,
  \end{align*}
where we have defined 
  \begin{align*}
    \theta(q_{\tau+1}):=\sum_{v\subseteq I_\tau}\tilde{\gamma}_{\tau+1,v\cup \{\tau+1\}} \sum_{\substack{(\bsk_v,k_{\tau+1}) \in \nat^{|v|+1}\\ (\bsk_v,k_{\tau+1},\bszero)\in D^*_{(\bsq^{*}_{\tau},q_{\tau+1}),p}}}\psi(\bsk_v,k_{\tau+1},\bszero) .
  \end{align*}
Following the same argument with the proof of Theorem \ref{theorem1}, we can obtain
  \begin{align*}
    \theta^\lambda(q^*_{\tau+1}) & \le \frac{1}{2^m-1}\sum_{v\subseteq I_\tau}\tilde{\gamma}_{\tau+1,v\cup \{\tau+1\}}^\lambda \frac{1}{(2^{2\lambda}-2)^{|v|+1}} .
  \end{align*}
Then using (\ref{eq:jensen}), we have
  \begin{align*}
    & B^\lambda(\bsq^{*}_{\tau+1},\bsgamma) \\
    = & \left(\left(1+\frac{\gamma_{\tau+1}}{3}\right)B(\bsq^{*}_\tau,\bsgamma)+\theta(q^*_{\tau+1})\right)^\lambda \\
    \le & \left(1+\frac{\gamma_{\tau+1}}{3}\right)^{\lambda}B^\lambda(\bsq^{*}_\tau,\bsgamma)+\theta^\lambda(q^*_{\tau+1}) \\
    = & \left(1+\frac{\gamma_{\tau+1}}{3}\right)^{\lambda}\frac{1}{2^m-1}\sum_{\emptyset \ne v\subseteq I_\tau}\tilde{\gamma}_{\tau,v}^\lambda \frac{1}{(2^{2\lambda}-2)^{|v|}}+\frac{1}{2^m-1}\sum_{v\subseteq I_\tau}\tilde{\gamma}_{\tau+1,v\cup \{\tau+1\}}^\lambda \frac{1}{(2^{2\lambda}-2)^{|v|+1}} \\
    = & \frac{1}{2^m-1}\sum_{\emptyset \ne v\subseteq I_\tau}\tilde{\gamma}_{\tau+1,v}^\lambda \frac{1}{(2^{2\lambda}-2)^{|v|}}+\frac{1}{2^m-1}\sum_{v\subseteq I_\tau}\tilde{\gamma}_{\tau+1,v\cup \{\tau+1\}}^\lambda \frac{1}{(2^{2\lambda}-2)^{|v|+1}} \\
    = & \frac{1}{2^m-1}\sum_{\emptyset \ne v\subseteq I_{\tau+1}}\tilde{\gamma}_{\tau+1,v}^\lambda \frac{1}{(2^{2\lambda}-2)^{|v|}} ,
  \end{align*}
for $1/2<\lambda \le 1$. Hence the result follows.
\end{proof}

\begin{remark}
For any $\tau$ such that $1\le \tau\le s$, we have
  \begin{align*}
    B(\bsq^{*}_\tau,\bsgamma) & \le \frac{1}{(2^m-1)^{1/\lambda}}\left[ \sum_{\emptyset \ne v\subseteq I_\tau}\prod_{j\in v}\left(\frac{\gamma_j}{3}\right)^\lambda \prod_{j'\in I_\tau\setminus v}\left(1+\frac{\gamma_{j'}}{3}\right)^\lambda \frac{1}{(2^{2\lambda}-2)^{|v|}} \right]^{1/\lambda} \\
    & = \frac{1}{(2^m-1)^{1/\lambda}}\left[ \prod_{j=1}^{\tau}\left(\left( \frac{1}{2^{2\lambda}-2}\cdot \frac{\gamma_j}{3}\right)^\lambda+\left(1+\frac{\gamma_j}{3}\right)^\lambda\right)-\prod_{j=1}^{\tau}\left(1+\frac{\gamma_j}{3}\right)^\lambda \right]^{1/\lambda} \\
    & \le \frac{1}{(2^m-1)^{1/\lambda}}\left[ \prod_{j=1}^{\tau}\left(1+\frac{2^{2\lambda}-1}{2^{2\lambda}-2}\left( \frac{\gamma_j}{3}\right)^\lambda\right)-\prod_{j=1}^{\tau}\left(1+\frac{\gamma_j}{3}\right)^\lambda \right]^{1/\lambda}, 
  \end{align*}
where we have used (\ref{eq:jensen}) in the last inequality. This expression gives an upper bound on $B(\bsq^{*}_\tau,\bsgamma)$ using not $\tilde{\gamma}_{\tau,v}$ but $\gamma_j$ for $1\le j\le s$. Then as in Corollary \ref{cor:tractability}, assume that the sequence of weights $\gamma_1,\gamma_2,\ldots,$ satisfies the condition
  \begin{align*}
    B_{\bsgamma,\lambda} := \sup_{s\in \nat}\frac{\left[ \prod_{j=1}^{s}\left(1+\frac{2^{2\lambda}-1}{2^{2\lambda}-2}\left( \frac{\gamma_j}{3}\right)^\lambda\right)-\prod_{j=1}^{s}\left(1+\frac{\gamma_j}{3}\right)^\lambda \right]^{1/\lambda}}{\prod_{j=1}^{s}\left(1+\frac{\gamma_j}{3}\right)-1},
  \end{align*}
for some $\lambda$ such that $1/2<\lambda \le 1$. Then the mean square weighted $\cL_2$ discrepancy is strongly tractable with the exponent of tractability at most $\lambda$.
\end{remark}

\begin{remark}\label{remark:fast_cbc}
The criterion used in Algorithm \ref{cbc_algor_product} has been simplified as
  \begin{align*}
    B((\bsq^{*}_{\tau-1},q_\tau),\bsgamma) = -\prod_{j=1}^{\tau}\left(1+\frac{\gamma_{j}}{3}\right)+\frac{1}{2^m}\sum_{n=0}^{2^m-1}\prod_{j=1}^{\tau}\left[1+\gamma_{j}\phi (x_{n,j})\right] ,
  \end{align*}
see Algorithm \ref{cbc_algor_product}. For this form, it is possible to reduce the computational cost of the CBC construction by using the fast Fourier transform as shown in \cite{NC06a,NC06b}. 
\end{remark}

%%%%%%%%%%%%%%%%%%%%%%%%%%%%%%%%%%%%%%%%%%%%%%%%%%%%%%%%%%%
%%%%%%%%%%%%%%%%%%%%%%%%%%%%%%%%%%%%%%%%%%%%%%%%%%%%%%%%%%%
%%%%%%%%%%%%%%%%%%%%%%%%%%%%%%%%%%%%%%%%%%%%%%%%%%%%%%%%%%%
\section{Numerical experiments}\label{numer}

Finally, we demonstrate the performance of our constructed polynomial lattice rules. We focus on the case of product weights, that is, $\gamma_u=\prod_{j\in u}\gamma_j$, because of their importance in practice and the availability of the fast CBC construction algorithm using the fast Fourier transform as mentioned in Remark \ref{remark:fast_cbc}. Three choices for $\gamma_j$ are considered here: $\gamma_j=1$ (unweighted), $\gamma_j=0.9^j$ and $\gamma_j=1/j^2$ for $j=1,\ldots,s$.

We compare the performance of our constructed polynomial lattice point sets with that of Sobol' sequences, which is one of the most well-known digital sequences over $\FF_2$ \cite{Sob67}. Since the weights emphasize the relative importance of the discrepancy of lower dimensional projections, we use Sobol' sequences as constructed in \cite{JK08}, which should work as a good competitor.

In Table \ref{tb:1}-\ref{tb:3}, we show the values of the mean square weighted $\cL_2$ discrepancy for Sobol' sequences and our constructed polynomial lattice point sets, denoted by Sobol' and PLR respectively, with $m=4,\ldots, 15$ and $s=1,5,50,100$ and different choices for the weights.

As expected from Remark \ref{remark:cbc_one_dimension}, we obtain exactly the same values for both the rules for $s=1$ and achieve the optimal rate of convergence, $2^{-2m}$, independent of the choice of the weights. In case of $s=5$, although Sobol' sequence provide the slightly better results for large $m$, the values are comparable. For $s=50$ and $s=100$, we obtain almost the same values for both the rules in the unweighted case, while our constructed polynomial lattice point sets outperform Sobol' sequences in other cases.
\begin{table}
\caption{The mean square weighted $\cL_2$ discrepancy for $\gamma_j=1$.}
\label{tb:1}
\begin{tabular}{c||cc|cc|cc|cc}
\hline
$m$ & \multicolumn{2}{|c|}{$s=1$} & \multicolumn{2}{|c|}{$s=5$} & \multicolumn{2}{|c|}{$s=50$} & \multicolumn{2}{|c}{$s=100$} \\ \hline
  & Sobol' & PLR & Sobol' & PLR & Sobol' & PLR & Sobol' & PLR \\ \hline
4	& 6.51E-04	& 6.51E-04	& 4.83E-02	& 3.79E-02	& 3.93E+07	& 3.91E+07	& 2.54E+16	& 2.54E+16 \\
5	& 1.63E-04	& 1.63E-04	& 1.45E-02	& 1.37E-02	& 1.96E+07	& 1.94E+07	& 1.27E+16	& 1.27E+16 \\
6	& 4.07E-05	& 4.07E-05	& 5.04E-03	& 4.29E-03	& 9.70E+06	& 9.64E+06	& 6.35E+15	& 6.35E+15 \\
7	& 1.02E-05	& 1.02E-05	& 1.27E-03	& 1.32E-03	& 4.78E+06	& 4.77E+06	& 3.18E+15	& 3.18E+15 \\
8	& 2.54E-06	& 2.54E-06	& 4.11E-04	& 4.69E-04	& 2.36E+06	& 2.35E+06	& 1.59E+15	& 1.59E+15 \\
9	& 6.36E-07	& 6.36E-07	& 1.21E-04	& 1.38E-04	& 1.17E+06	& 1.16E+06	& 7.94E+14	& 7.94E+14 \\
10	& 1.59E-07	& 1.59E-07	& 4.01E-05	& 4.47E-05	& 5.80E+05	& 5.70E+05	& 3.97E+14	& 3.97E+14 \\
11	& 3.97E-08	& 3.97E-08	& 1.15E-05	& 1.28E-05	& 2.89E+05	& 2.80E+05	& 1.98E+14	& 1.98E+14 \\
12	& 9.93E-09	& 9.93E-09	& 3.45E-06	& 4.41E-06	& 1.44E+05	& 1.37E+05	& 9.92E+13	& 9.91E+13 \\
13	& 2.48E-09	& 2.48E-09	& 1.17E-06	& 1.39E-06	& 7.17E+04	& 6.69E+04	& 4.96E+13	& 4.95E+13 \\
14	& 6.21E-10	& 6.21E-10	& 2.78E-07	& 4.05E-07	& 3.56E+04	& 3.27E+04	& 2.48E+13	& 2.48E+13 \\
15	& 1.55E-10	& 1.55E-10	& 7.98E-08	& 1.31E-07	& 1.76E+04	& 1.59E+04	& 1.24E+13	& 1.24E+13 \\
\hline                                                                                 
\end{tabular}
\end{table}

\begin{table}
\caption{The mean square weighted $\cL_2$ discrepancy for $\gamma_j=0.9^j$.}
\label{tb:2}
\begin{tabular}{c||cc|cc|cc|cc}
\hline
$m$ & \multicolumn{2}{|c|}{$s=1$} & \multicolumn{2}{|c|}{$s=5$} & \multicolumn{2}{|c|}{$s=50$} & \multicolumn{2}{|c}{$s=100$} \\ \hline
  & Sobol' & PLR & Sobol' & PLR & Sobol' & PLR & Sobol' & PLR \\ \hline
4	& 5.86E-04	& 5.86E-04	& 2.13E-02	& 1.72E-02	& 1.43E+00	& 1.22E+00	& 1.48E+00	& 1.26E+00 \\
5	& 1.46E-04	& 1.46E-04	& 6.25E-03	& 5.93E-03	& 6.27E-01	& 5.16E-01	& 6.47E-01	& 5.34E-01 \\
6	& 3.66E-05	& 3.66E-05	& 2.07E-03	& 1.80E-03	& 2.47E-01	& 2.17E-01	& 2.56E-01	& 2.25E-01 \\
7	& 9.16E-06	& 9.16E-06	& 5.25E-04	& 5.41E-04	& 9.81E-02	& 8.85E-02	& 1.02E-01	& 9.19E-02 \\
8	& 2.29E-06	& 2.29E-06	& 1.64E-04	& 1.84E-04	& 3.94E-02	& 3.52E-02	& 4.11E-02	& 3.67E-02 \\
9	& 5.72E-07	& 5.72E-07	& 4.73E-05	& 5.23E-05	& 1.60E-02	& 1.41E-02	& 1.66E-02	& 1.47E-02 \\
10	& 1.43E-07	& 1.43E-07	& 1.52E-05	& 1.70E-05	& 6.73E-03	& 5.62E-03	& 7.02E-03	& 5.87E-03 \\
11	& 3.58E-08	& 3.58E-08	& 4.29E-06	& 5.19E-06	& 2.97E-03	& 2.26E-03	& 3.10E-03	& 2.36E-03 \\
12	& 8.94E-09	& 8.94E-09	& 1.25E-06	& 1.58E-06	& 1.25E-03	& 8.90E-04	& 1.31E-03	& 9.33E-04 \\
13	& 2.24E-09	& 2.24E-09	& 4.01E-07	& 4.85E-07	& 5.61E-04	& 3.57E-04	& 5.86E-04	& 3.75E-04 \\
14	& 5.59E-10	& 5.59E-10	& 9.89E-08	& 1.43E-07	& 2.13E-04	& 1.41E-04	& 2.24E-04	& 1.49E-04 \\
15	& 1.40E-10	& 1.40E-10	& 2.79E-08	& 4.38E-08	& 7.84E-05	& 5.61E-05	& 8.30E-05	& 5.91E-05 \\
\hline                                                                                           
\end{tabular}
\end{table}

\begin{table}
\caption{The mean square weighted $\cL_2$ discrepancy for $\gamma_j=1/j^2$.}
\label{tb:3}
\begin{tabular}{c||cc|cc|cc|cc}
\hline
$m$ & \multicolumn{2}{|c|}{$s=1$} & \multicolumn{2}{|c|}{$s=5$} & \multicolumn{2}{|c|}{$s=50$} & \multicolumn{2}{|c}{$s=100$} \\ \hline
  & Sobol' & PLR & Sobol' & PLR & Sobol' & PLR & Sobol' & PLR \\ \hline
4	& 6.51E-04	& 6.51E-04	& 1.84E-03	& 1.73E-03	& 2.99E-03	& 2.47E-03	& 3.07E-03	& 2.53E-03 \\
5	& 1.63E-04	& 1.63E-04	& 4.81E-04	& 4.76E-04	& 8.63E-04	& 7.31E-04	& 8.95E-04	& 7.50E-04 \\
6	& 4.07E-05	& 4.07E-05	& 1.35E-04	& 1.28E-04	& 2.64E-04	& 2.10E-04	& 2.78E-04	& 2.17E-04 \\
7	& 1.02E-05	& 1.02E-05	& 3.53E-05	& 3.43E-05	& 7.42E-05	& 5.98E-05	& 8.09E-05	& 6.24E-05 \\
8	& 2.54E-06	& 2.54E-06	& 9.21E-06	& 9.43E-06	& 2.23E-05	& 1.75E-05	& 2.48E-05	& 1.84E-05 \\
9	& 6.36E-07	& 6.36E-07	& 2.53E-06	& 2.51E-06	& 6.56E-06	& 4.94E-06	& 7.37E-06	& 5.24E-06 \\
10	& 1.59E-07	& 1.59E-07	& 6.94E-07	& 6.86E-07	& 1.75E-06	& 1.41E-06	& 2.02E-06	& 1.51E-06 \\
11	& 3.97E-08	& 3.97E-08	& 1.82E-07	& 1.90E-07	& 4.87E-07	& 4.12E-07	& 5.53E-07	& 4.43E-07 \\
12	& 9.93E-09	& 9.93E-09	& 4.76E-08	& 5.00E-08	& 1.39E-07	& 1.16E-07	& 1.62E-07	& 1.26E-07 \\
13	& 2.48E-09	& 2.48E-09	& 1.29E-08	& 1.35E-08	& 4.06E-08	& 3.40E-08	& 4.89E-08	& 3.70E-08 \\
14	& 6.21E-10	& 6.21E-10	& 3.35E-09	& 3.80E-09	& 1.29E-08	& 1.01E-08	& 1.53E-08	& 1.10E-08 \\
15	& 1.55E-10	& 1.55E-10	& 8.87E-10	& 1.01E-09	& 3.61E-09	& 2.97E-09	& 4.37E-09	& 3.27E-09 \\
\hline                                                                                           
\end{tabular}
\end{table}

\end{document}